\documentclass[10pt, reqno]{amsart}
\usepackage[dvipsnames]{xcolor}
\usepackage{graphicx}
\usepackage{amsmath}
\usepackage{amsfonts}
\usepackage{amssymb}
\usepackage{bbm}
\usepackage{amsthm}
\usepackage{tikz}
\usepackage[sc]{mathpazo}
\usepackage[T1]{fontenc} 
\usepackage{bm}

\usepackage{enumitem}
    \setenumerate{itemsep=5pt} % adjust space between list items
    \setitemize{itemsep=5pt} % adjust space between list items

\theoremstyle{plain}
\newtheorem{theorem}{Theorem}

\newtheorem{lemma}[theorem]{Lemma}
\newtheorem{corollary}[theorem]{Corollary}

\theoremstyle{definition}
\newtheorem{definition}[theorem]{Definition}
\newtheorem{problem}[theorem]{Problem}

\newtheorem{remark}[theorem]{Remark}
\newtheorem{example}[theorem]{Example}

\numberwithin{equation}{section}

\newcommand{\heading}[1]{\bigskip\noindent\textsc{\underline{#1}.}}

\newcommand{\C}{\mathbb{C}}
\newcommand{\F}{\mathcal{F}}
\newcommand{\R}{\mathbb{R}}
\newcommand{\E}{\mathbf{E}}
\newcommand{\M}{\mathrm{M}}
\newcommand{\N}{\mathbb{N}}
\newcommand{\NN}{\mathfrak{N}}
\newcommand{\OO}{\mathrm{O}}
\renewcommand{\P}{\mathbf{P}}
\newcommand{\V}{\mathcal{V}}
\newcommand{\W}{\mathcal{W}}
\newcommand{\h}{\mathrm{H}}
\newcommand{\LL}{\vec{\lambda}}
\def\mx{\langle x,x^*\rangle}
\newcommand{\TT}{\mathrm{T}}

\newcommand{\X}{\vec{X}}
\newcommand{\tr}{\operatorname{tr}}
\newcommand{\diag}{\operatorname{diag}}
\newcommand{\ran}{\operatorname{ran}}
\renewcommand{\vec}[1]{{\bm{#1}}}
\newcommand{\norm}[1]{\| #1 \|}
\newcommand{\cnorm}[1]{|\!|\!| #1 |\!|\!|}

\allowdisplaybreaks
\usepackage{hyperref}

\begin{document}

\title[Norms on Complex Matrices Induced by Random Vectors]{Norms on Complex Matrices Induced by Random Vectors }
\author[\'A. Ch\'avez]{\'Angel Ch\'avez}
\address{Department of Mathematics and Statistics, Pomona College, 610 N. College Ave., Claremont, CA 91711} 
	\email{angel.chavez@pomona.edu}

\author[S.R.~Garcia]{Stephan Ramon Garcia}
\email{stephan.garcia@pomona.edu}
\urladdr{\url{http://pages.pomona.edu/~sg064747}}

\author[J.~Hurley]{Jackson Hurley}
\email{jacksonwhurley@gmail.com}

\thanks{SRG partially supported by NSF grant DMS-2054002.}

\begin{abstract}
We introduce a family of norms on the $n \times n$ complex matrices. These norms arise from a probabilistic framework, and their construction and validation involve probability theory, partition combinatorics, and trace polynomials in noncommuting variables. As a consequence, we obtain a generalization of Hunter's positivity theorem for the complete homogeneous symmetric polynomials.
\end{abstract}

\maketitle

\section{Introduction}
This paper introduces norms on the space $\M_n$ of $n\times n$ complex matrices that are induced by random vectors in $\R^n$. Specifically, we construct a family of norms for each random vector $\X$ whose entries are identically distributed (iid) random variables with sufficiently many moments. Initially, these norms are defined on complex Hermitian matrices as symmetric functions of their (necessarily real) eigenvalues.  This contrasts with Schatten and Ky-Fan norms, which are defined in terms of singular values.  To be more specific, our norms do not arise from the machinery of symmetric gauge functions \cite[Sect.~7.4.7]{HJ}.

The random vector norms we construct are actually generalized versions of the complete homogeneous symmetric (CHS) polynomial norms introduced in \cite{Aguilar}. 

\subsection{Preliminaries}
Our main result (Theorem \ref{Theorem:Main} on p.~\pageref{Theorem:Main}) connects a wide range of topics,
such as cumulants, Bell polynomials, partitions, and Schur convexity,
which may not all be familiar to a given reader.  Consequently, we briefly cover the preliminary concepts and notation necessary to state our main results.

%%%%%%%%%%%%%%%%%%%%%%%%%%%%%%%%%%%%%
\heading{Numbers and Matrices}
In what follows, $\N = \{1,2,\ldots\}$; the symbols $\R$ and $\C$ denote the real and complex number systems, respectively.
Let $\M_n$ denote the set of $n \times n$ complex matrices and
$\h_n \subset \M_n$ the subset of $n\times n$ Hermitian complex matrices.
We reserve the letter $A$ for Hermitian matrices (so $A=A^*$) and $Z$ for arbitrary square complex matrices.
The eigenvalues of each $A\in \h_n$ are real and denoted
$\lambda_1(A)\geq \lambda_2(A)\geq \cdots \geq \lambda_n(A)$. 
We may write $\lambda_1, \lambda_2, \ldots, \lambda_n$ if $A$ is understood and we denote by
$\LL=(\lambda_1, \lambda_2, \ldots, \lambda_n)$ the vector of  eigenvalues of $A$. 

%%%%%%%%%%%%%%%%%%%%%%%%%%%%%%%%%%%%%
\heading{Probability theory}
A \emph{probability space} is a measure space $(\Omega, \F, \P)$, in which 
$\F$ is a $\sigma$-algebra on $\Omega$, $\P$ is nonnegative, and $\P(\Omega)=1$.
A \emph{random variable} is a measurable function $X: \Omega\to \R$. 
We assume that $\Omega \subseteq \R$ and $X$ is nondegenerate, that is, nonconstant. 
The \emph{expectation} of $X$ is 
$\E  [X]=\int_{\Omega} X \,d\P$, often written as $\E X$.
For $p\geq 1$, let $L^p(\Omega, \F, \P)$ denote the vector space of random variables
such that $\norm{X}_{L^p}=(\E |X|^p)^{1/p} < \infty$.
The pushforward measure $X_*\P$ of $X$
is the \emph{probability distribution} of $X$.  
The \emph{cumulative distribution} of $X$ is $F_X(x)=\P(X\leq x)$, which is the pushforward measure of $(-\infty, x]$. 
If $X_{*}\P$ is absolutely continuous with respect to Lebesgue measure $m$, then
the Radon--Nikodym derivative
$f_X= dX_* P/dm$ is the \emph{probability density function} (PDF) of $X$.  
See \cite[Ch.~1]{Billingsley} for background.

%%%%%%%%%%%%%%%%%%%%%%%%%%%%%%%
\heading{Random vectors}
A \emph{random vector} is a tuple $\X=(X_1, X_2, \ldots, X_n)$, in which 
$X_1, X_2, \ldots, X_n$ are real-valued random variables on a common probability space $(\Omega,\F,\P)$;
we assume $\Omega \subseteq \R$.
A random vector $\X$ is \emph{positive definite} if its \emph{second-moment matrix} 
$\Sigma(\X) = [\E X_iX_j ]_{i,j=1}^n$ exists and is positive definite. This occurs if the $X_i$ are identically
distributed and belong to $L^2(\Omega,\F,\P)$; see Lemma \ref{Lemma:PosDef}.

%%%%%%%%%%%%%%%%%%%%%%%%%%%%%%%
\heading{Moments}
For $k \in \N$, the $k$th \emph{moment} of $X$ is $\mu_k = \E[X^k]$, if it exists.
If $X$ has PDF $f_X$, then $\mu_k = \int_{-\infty}^{\infty} x^k f_X(x)\,dm(x)$.
The \emph{mean} of $X$ is $\mu_1$ and 
the \emph{variance} of $X$ is $\mu_2 - \mu_1^2$; Jensen's inequality ensures that 
the variance is positive since $X$ is nondegenerate.
The \emph{moment generating function} (if it exists) of $X$ is
\begin{equation}\label{eq:MGF}
M(t)=\E [e^{tX}]=\sum_{k=0}^{\infty} \E [X^k] \frac{t^k}{k!} = \sum_{k=0}^{\infty} \mu_k\frac{t^k}{k!}.
\end{equation}
If $X_1, X_2, \ldots, X_n$ are independent, then 
$\E [X_1^{i_1} X_2^{i_2}\cdots X_n^{i_n}]=\prod_{k=1}^n\E  [X_k^{i_k}]$
for all $i_1, i_2, \ldots, i_n \in \N$ whenever both sides exist.

%%%%%%%%%%%%%%%%%%%%%%%%%%%%%%%
\heading{Cumulants}
If $X$ admits a moment generating function $M(t)$, 
then the $r$th \emph{cumulant} $\kappa_r$ of $X$ is defined by the \emph{cumulant generating function} 
\begin{equation}\label{eq:CumulantGen}
K(t)=\log M(t)=\sum_{r=1}^{\infty} \kappa_r \frac{t^r}{r!}.
\end{equation} 
The first two cumulants are $\kappa_1 = \mu_1$ and $\kappa_2 = \mu_2 - \mu_1^2$.
If $X$ does not admit a moment generating function but $X\in L^d(\Omega, \mathcal{F}, \P)$ for some  
$d\in \N$, we can define $\kappa_1, \kappa_2, \ldots, \kappa_d$ by the recursion
$\mu_r=\sum_{\ell=0}^{r-1}{r-1\choose \ell} \mu_{\ell}\kappa_{r-\ell}$
for $1 \leq r \leq d$; see \cite[Sec.~9]{Billingsley}.

%%%%%%%%%%%%%%%%%%%%%%%%%%%%%%%
\heading{Power-series coefficients}
The coefficient $c_k$ of $t^k$ in $f(t) = \sum_{r=0}^{\infty} c_r t^r$ is denoted $[t^k]f(t)$,
as is standard in combinatorics and the study of generating functions.

%%%%%%%%%%%%%%%%%%%%%%%%%%%%%%%
\heading{Complete Bell polynomials}
The \emph{complete Bell polynomials of degree $\ell$} \cite[Sec.~II]{Bell} are the polynomials $B_{\ell}(x_1, x_2, \ldots, x_{\ell})$ defined by 
\begin{equation}\label{eq:ExpoBell}
\sum_{\ell=0}^{\infty} B_{\ell}(x_1, x_2, \ldots, x_{\ell}) \frac{t^{\ell}}{\ell !}=\exp\bigg( \sum_{j=1}^{\infty} x_j \frac{t^j}{j!}\bigg).
\end{equation} 
The first several even-degree complete Bell polynomials are 
\begin{align}
B_0 &= 1,  \nonumber\\
B_2(x_1,x_2)&=x_1^2+x_2, \quad \text{and} \nonumber\\
B_4(x_1, x_2, x_3, x_4)&=x_1^4+6x_1^2x_2+4x_1x_3+3x_2^2+x_4. \label{eq:Bell4}
\end{align}

%%%%%%%%%%%%%%%%%%%%%%%%%%%%%%%%%
\heading{Symmetric and positive functions} 
A function is \emph{symmetric} if it is invariant under all permutations of its arguments. 
A continuous real-valued function on $\M_n$ or $\h_n$ is \emph{positive definite} if it is everywhere 
positive, except perhaps at $0$.

%%%%%%%%%%%%%%%%%%%%%%%%%%%%%%%%%
\heading{Partitions} 
A \emph{partition} of $d\in \N$ is a tuple $\vec{\pi}=(\pi_1, \pi_2, \ldots, \pi_r) \in \N^r$ such that
$\pi_1 \geq \pi_2 \geq \cdots \geq \pi_r$ and
$\pi_1+ \pi_2 + \cdots + \pi_r = d$ \cite[Sec.~1.7]{StanleyBook1}.  We denote this 
$\vec{\pi} \vdash d$ and write $| \vec{\pi}| = r$ for the number of parts in the partition.
Define
\begin{equation}\label{eq:DefY}
\kappa_{\vec{\pi}} = \kappa_{\pi_1} \kappa_{\pi_2} \cdots \kappa_{\pi_{\ell}}
\quad \text{and} \quad
y_{\vec{\pi}}=\prod_{i\geq 1}(i!)^{m_i}m_i!,
\end{equation} 
in which $m_i=m_i(\vec{\pi})$ is the multiplicity of $i$ in $\vec{\pi}$.
For example, 
$\vec{\pi} = (4,4,2,1,1,1)$ yields $\kappa_{\vec{\pi}} = \kappa_4^2 \kappa_2 \kappa_1^3$ and
$y_{\vec{\pi}}= (1!^3 3!) (2!^1 1!) (4!^2 2!) = 13{,}824$.
Note that $y_{\vec{\pi}}$ is not the quantity $z_{\vec{\pi}} = \prod_{i \geq 1} i^{m_i} m_i!$ 
from symmetric function theory \cite[Prop.~7.7.6]{StanleyBook2}.

%%%%%%%%%%%%%%%%%%%%%%%%%%%%%%%%%
\heading{Power sums} 
For $\vec{\pi} \vdash d$, let
$p_{\vec{\pi}}(x_1, x_2, \ldots, x_n)=p_{\pi_1}p_{\pi_1}\cdots p_{\pi_r}$,
in which 
\begin{equation*}
p_k(x_1,x_2, \ldots, x_n)=x_1^k+x_2^k+\cdots +x_n^k
\end{equation*}
is a \emph{power-sum symmetric polynomial};
we simply write $p_k$ if the variables are clear from context.
If $A \in \h_n$ has eigenvalues $\vec{\lambda} = (\lambda_1,\lambda_2,\ldots,\lambda_n)$, then we write
\begin{equation}\label{eq:pTrace}
p_{\vec{\pi}}(\vec{\lambda}) =p_{\pi_1}(\vec{\lambda})p_{\pi_2}(\vec{\lambda})\cdots p_{\pi_r}(\vec{\lambda})
=(\tr A^{\pi_1})(\tr A^{\pi_2})\cdots (\tr A^{\pi_{r}}).
\end{equation} 

%%%%%%%%%%%%%%%%%%%%%%%%%%%%%%%%%
\heading{Complete homogeneous symmetric polynomials} 
The \emph{complete homogeneous symmetric polynomial} (CHS) of degree $d$ 
in $x_1, x_2, \ldots x_n$ is
\begin{equation}\label{eq:CHS}
h_d(x_1,x_2,\ldots,x_n) = \sum_{1 \leq i_1 \leq \cdots \leq i_{d} \leq n} x_{i_1} x_{i_2}\cdots x_{i_d},
\end{equation}
the sum of all monomials of degree $d$ in $x_1,x_2,\ldots,x_n$;
 \cite[Sec.~7.5]{StanleyBook2}.   For example,\label{p:CHS}
\begin{align*}
h_0(x_1,x_2) &=1, \\
h_2(x_1,x_2)&= x_1^2+x_1 x_2+x_2^2, \quad\text{and}\\
h_4(x_1,x_2)&= x_1^4  + x_1^3 x_2 + x_1^2 x_2^2 + x_1 x_2^3 + x_2^4.
\end{align*}
Hunter proved that the even-degree 
complete homogeneous symmetric (CHS) polynomials are positive definite \cite{Hunter}.
This has been rediscovered many times
\cite[Thm.~1]{Aguilar},
\cite[Lem.~3.1]{Barvinok},
\cite{Baston},
\cite[Thm.~2]{BGON},
\cite[Cor.~17]{GOOY},
\cite[Thm.~2.3]{Roventa},
and \cite[Thm.~1]{Tao}.

%%%%%%%%%%%%%%%%%%%%%%%%%%%%%%%%%
\heading{Schur convexity}
Let $\widetilde{\vec{x}}=(\widetilde{x}_1, \widetilde{x}_2, \ldots, \widetilde{x}_n)$ be the 
nondecreasing rearrangement of $\vec{x}=(x_1,x_2, \ldots, x_n) \in \R^n$. 
Then $\vec{y}$  \emph{majorizes} $\vec{x}$, denoted $\vec{x}\prec \vec{y}$, if 
\begin{equation*}
\sum_{i=1}^n \widetilde{x}_i = \sum_{i=1}^n \widetilde{y}_i
\quad \text{and} \quad
\sum_{i=1}^k \widetilde{x}_i \leq \sum_{i=1}^k \widetilde{y}_i
\quad \text{for $1 \leq k \leq n$}.
\end{equation*}
A function $f:\R^n\to \R$ is \emph{Schur convex} if  $f(\vec{x})\leq f(\vec{y})$ whenever $\vec{x}\prec \vec{y}$. 
A symmetric function $f$ is Schur convex if and only if 
\begin{equation*}
(x_i-x_j)\Big( \frac{\partial}{\partial x_i}-\frac{\partial}{\partial x_j} \Big)f(x_1, x_2, \ldots, x_n)\geq 0
\quad\text{for all $1\leq i<j\leq n$},
\end{equation*}
with equality if and only if $x_i=x_j$ \cite[p.~259]{Roberts}.

%%%%%%%%%%%%%%%%%%%%%%%%%%%%%%%%%%%%%
\subsection{Statement of Main Results}
With the preliminary concepts and notation covered, 
we are now in a position to state our main theorem.
In what follows, $\Gamma$ is the gamma function and $\langle \cdot, \cdot\rangle$ is the Euclidean inner product on $\R^n$.

\begin{theorem}\label{Theorem:Main}
Let $d\geq 2$ and $\X=(X_1, X_2, \ldots, X_n)$, in which
$X_1, X_2, \ldots, X_n \in L^d(\Omega,\F,\P)$ are nondegenerate independent and identically distributed random variables.
\begin{enumerate}[leftmargin=*]
\item $\cnorm{A}_{\X,d}= \bigg(\dfrac{  \E |\langle \X, \LL\rangle|^d}{\Gamma(d+1)} \bigg)^{1/d}$ is a norm on $\h_n$.

\item If the $X_i$ admit a moment generating function $M(t)$ and $d \geq 2$ is even, then
\begin{equation}\label{eq:MainMGF}
\cnorm{A}_{\X,d}^d = [t^d] M_{\Lambda}(t)
\quad \text{for all $A \in \h_n$},
\end{equation}
in which $M_{\Lambda}(t) = \prod_{i=1}^n M(\lambda_i t)$ is the moment generating
function for the random variable 
$\Lambda =\langle \X, \LL(A) \rangle=\lambda_1X_1+\lambda_2X_2+\cdots +\lambda_n X_n$.
In particular, $\cnorm{A}_{\X,d}$ is a positive definite, homogeneous,
symmetric polynomial in the eigenvalues of $A$.

\item If the first $d$ moments of $X_i$ exist, then
\begin{align}
\cnorm{A}_{\X,d}^d
&= \frac{1}{d!} B_{d}(\kappa_1\tr A, \kappa_2\tr A^2, \ldots, \kappa_d\tr A^d) \label{eq:MainBell} \\
&= \sum_{\vec{\pi}\vdash d}\frac{\kappa_{\vec{\pi}}p_{\vec{\pi}} (\vec{\lambda})}{y_{\vec{\pi}}}
\quad \text{for $A \in \h_n$}, \label{eq:RealPermForm}
\end{align}
in which $B_d$ is given by \eqref{eq:ExpoBell},
and in which $\kappa_{\vec{\pi}}$ and $y_{\vec{\pi}}$ are defined in \eqref{eq:DefY},
$p_{\vec{\pi}} (\vec{\lambda})$ is defined in \eqref{eq:pTrace},
and the second sum runs over all partitions $\vec{\pi}$ of $d$. 

\item The function $\vec{\lambda}(A) \mapsto \cnorm{A}_{\X,d}$ is Schur convex.

\item Let $\vec{\pi}=(\pi_1, \pi_2, \ldots,\pi_r)$ be a partition of $d$. 
Define $\TT_{\vec{\vec{\pi}}} : \M_{n}\to \R$
by setting $\TT_{\vec{\pi}}(Z)$ to be $1/{d\choose d/2}$ times the sum over the $\binom{d}{d/2}$ 
possible locations to place $d/2$ adjoints ${}^*$ among the $d$ copies of $Z$ in
\begin{equation*}
(\tr \underbrace{ZZ\cdots Z}_{\pi_1})
(\tr \underbrace{ZZ\cdots Z}_{\pi_2})
\cdots
(\tr \underbrace{ZZ\cdots Z}_{\pi_r}).
\end{equation*}
Then
\begin{equation}\label{eq:Extended}
\cnorm{Z}_{\X,d}= \bigg( \sum_{\vec{\pi} \,\vdash\, d} \frac{ \kappa_{\vec{\pi}}\TT_{\vec{\pi}}(Z)}{y_{\vec{\pi}}}\bigg)^{1/d}
\quad \text{for $Z \in \M_n$},
\end{equation} 
in which $\kappa_{\vec{\pi}}$ and $y_{\vec{\pi}}$ are defined in \eqref{eq:DefY}
and the sum runs over all partitions $\vec{\pi}$ of $d$, 
is a norm on $\M_n$ that restricts to the norm on $\h_n$ above.  In particular, $\cnorm{Z}_{\X,d}^d$ 
is a positive definite trace polynomial in $Z$ and $Z^*$.
\end{enumerate}
\end{theorem}

The independence of the $X_i$ is not needed in (a) and (d);
see Remarks \ref{Remark:Conjugation} and \ref{Remark:SchurNot}, respectively.
A more precise definition of $\TT_{\vec{\pi}}(Z)$ is in Subsection \ref{Subsection:ProofE},
although the examples in the next section better illustrate how to compute \eqref{eq:Extended}.

The positive definiteness of \eqref{eq:MainMGF}, \eqref{eq:MainBell}, and \eqref{eq:Extended}
is guaranteed by Theorem \ref{Theorem:Main};
that these expressions satisfy the triangle inequality
is difficult to verify directly.
Positivity itself is not obvious since
we consider the eigenvalues of $A \in \h_n$ and not their absolute values in (a) and (b).
Thus, these norms on $\h_n$ do not arise from singular values or the theory of symmetric gauge functions \cite[Sect.~7.4.7]{HJ}.
Norms like ours can distinguish singularly cospectral graphs,
unlike the operator, Frobenius, Schatten--von Neumann, and Ky Fan norms; see \cite[Ex.~2]{Aguilar}.

\subsection{Organization}
This paper is organized as follows.
We first cover a range of examples and applications in Section \ref{Section:Examples}, including
a generalization of Hunter's positivity theorem.
The proof of Theorem \ref{Theorem:Main}, which is lengthy and involves a variety of ingredients, 
is contained in Section \ref{Section:Proof}.
We end this paper in Section \ref{Section:Questions} with a list of open questions that demand further exploration.

\medskip\noindent\textbf{Acknowledgments.}  We thank Bruce Sagan for a helpful comment about monomial symmetric functions.

%%%%%%%%%%%%%%%%%%%%%%%%%%%%
\section{Examples and Applications}\label{Section:Examples}
We begin with general computations for small $d$ (Subsection \ref{Subsection:Small}).
We examine Gamma random variables in Subsection \ref{Subsection:Gamma}.
A special case leads to a generalization of Hunter's positivity theorem 
(Subsection \ref{Subsection:Hunter}).
We examine norms arising from familiar probability distributions
in Subsections \ref{Subsection:Normal} -- \ref{Subsection:Pareto}.

%%%%%%%%%%%%%%%%%%%%%%%%%%%%%%%%%%%%%%%%%%%
\subsection{Generic computations}\label{Subsection:Small}
Let $\X=(X_1, X_2, \ldots, X_n)$, in which the $X_i$ are 
nondegenerate independent and identically distributed (iid) random variables
such that the stated cumulants and moments exist.
For $d=2$ and $4$, we obtain trace-polynomial representations of 
$\cnorm{Z}_{d}$ in terms of cumulants or moments.  This can also be done for $d=6,8,\ldots$, 
but we refrain from the exercise.

\begin{example}\label{Example:General2}
The two partitions of $d=2$ satisfy 
$\kappa_{(2)} = \kappa_2= \mu_2 - \mu_1^2$,
$\kappa_{(1,1)} = \kappa_1^2 = \mu_1^2$,
and $y_{(2)} = y_{(1,1)} = 2$.
There are $\binom{2}{1} = 2$ ways to place two adjoints ${}^*$ in a string of two $Z$s.  Therefore,
\begin{align*}
\TT_{(2)}(Z) &= \frac{1}{2} \big(\tr(Z^*Z)+\tr(ZZ^*) \big)  =  \tr(Z^*Z) \quad \text{and} \\
\TT_{(1,1)}(Z) &= \frac{1}{2}\big( (\tr Z^*)(\tr Z) +(\tr Z)(\tr Z^*) \big) =(\tr Z^*)(\tr Z) ,
\end{align*}
so
\begin{equation}\label{eq:cn2}
\cnorm{Z}_{\X,2}^2 
= \sum_{\vec{\pi} \,\vdash\, d} \frac{ \kappa_{\vec{\pi}}\TT_{\vec{\pi}}(A)  }{y_{\vec{\pi}}}
= \frac{\mu_2-\mu_1^2}{2} \tr(Z^*Z) + \frac{\mu_1^2}{2} (\tr Z^*)(\tr Z).
\end{equation}
If $\mu_1 = 0$ (mean zero), then $\cnorm{\cdot}_2$
is a nonzero multiple of the Frobenius norm since the variance $\mu_2-\mu_1^2$ 
is positive by nondegeneracy.  
As predicted by Theorem \ref{Theorem:Main}, 
the norm \eqref{eq:cn2} on $\M_n$ reduces to \eqref{eq:MainBell} on $\h_n$
since $B_2(x_1,x_2)=x_1^2+x_2$ and
{\small
\begin{equation*}
\cnorm{A}_{\X,2}^2
=\frac{1}{2}B_2( \kappa_1 \tr A, \kappa_2 \tr A^2)
= \frac{1}{2}\big[ (\kappa_1 \tr A)^2 + \kappa_2 \tr (A^2) \big] 
= \frac{\mu_2 - \mu_1^2}{2} \tr(A^2) + \frac{\mu_1^2}{2} (\tr A)^2,
\end{equation*}
}%
which agrees with \eqref{eq:cn2} if $Z = A = A^*$.
\end{example}

\begin{example}\label{Example:General4}
The five partitions of $d=4$ satisfy
\begin{equation*}
\kappa_{(4)} = \kappa_4,\quad
\kappa_{(3,1)} = \kappa_1 \kappa_3,\quad
\kappa_{(2,2)} = \kappa_2^2,\quad
\kappa_{(2,1,1)} = \kappa_2 \kappa_1^2,\quad 
\kappa_{(1,1,1,1)} = \kappa_1^4,
\end{equation*}
and
\begin{equation*}
y_{(4)} = 24,\quad
y_{(3,1)} = 6,\quad
y_{(2,2)} = 8,\quad
y_{(2,1,1)} = 4,\quad 
y_{(1,1,1,1)} = 24.
\end{equation*}
There are $\binom{4}{2} = 6$ ways to place two adjoints ${}^*$ in a string of four $Z$s.  
For example, 
\begin{align*}
6\TT_{(3,1)}(Z)
&= (\tr Z^*Z^*Z)(\tr Z) + (\tr Z^*ZZ^*)(\tr Z) + (\tr Z^*ZZ)(\tr Z^*) \\
&\qquad + (\tr ZZ^*Z^*)(\tr Z) +(\tr ZZ^*Z)(\tr Z^*) +(\tr ZZZ^*)(\tr Z^*)\\
&=3 \tr (Z^{*2}Z)(\tr Z) +3 (\tr Z^2 Z^*)(\tr Z^*).
\end{align*}
Summing over all five partitions yields the following norm on $\M_n$:
\begin{align}
\cnorm{Z}_{\X,4}^4
&=  \frac{1}{72}\big(
3 \kappa_1^4 (\tr Z^*)^2 (\tr Z)^2 
+3 \kappa_2 \kappa_1^2 (\tr Z^*)^2 \tr (Z^2) 
+3 \kappa_2 \kappa_1^2 \tr (Z^{*2})(\tr Z)^2  \nonumber \\
&\qquad +12 \kappa_2 \kappa_1^2  (\tr Z^*) (\tr Z^*Z) (\tr Z)
+6 \kappa_3 \kappa_1 \tr ( Z^{*2}Z)(\tr Z) \nonumber \\
&\qquad +6 \kappa_3 \kappa_1 \tr (Z^*) \tr (Z^* Z^2 ) 
+6 \kappa_2^2 (\tr  Z^*Z)^2
+3 \kappa_2^2 \tr (Z^2) \tr (Z^{*2}) \nonumber \\
&\qquad +2 \kappa_4 \tr (Z^2 Z^{*2})+\kappa_4 \tr (Z Z^* Z Z^*) \big). \label{eq:Z44}
\end{align}
If $Z = A\in \h_n$,
Theorem \ref{Theorem:Main}.c and \eqref{eq:Bell4} ensure that the above reduces to
\begin{equation*}
\frac{1}{24}\big(
\kappa _1^4 (\tr A)^4+6 \kappa _1^2 \kappa _2 \tr(A^2) (\tr A)^2
+4 \kappa _1 \kappa _3 \tr (A^3) \tr (A) 
+ 3 \kappa _2^2 \tr (A^2)^2+\kappa _4 \tr (A^4) \big).
\end{equation*}
\end{example}

%%%%%%%%%%%%%%%%%%%%%%%%%%%%%%%%%%%%%%%%%%%
\subsection{Gamma random variables}\label{Subsection:Gamma}
Let $\X=(X_1, X_2, \ldots, X_n)$, in which the $X_i$ are independent random variables with probability density
\begin{equation}\label{eq:Gamma}
f(t)=\begin{cases}
  \dfrac{1}{\beta^{\alpha} \Gamma(\alpha)} t^{\alpha - 1} e^{-t/\beta} & \text{if $t> 0$},\\[8pt] 
    0 & \text{if $t\leq 0$}.
    \end{cases}
\end{equation}  
Here $\alpha, \beta>0$ ($\alpha = k/2$ and $\beta = 2$ yield a chi-squared
random variable with $k$ degrees of freedom; $\alpha=\beta=1$ is the standard exponential distribution). Then
\begin{equation*}
M(t) = (1 - \beta t)^{-\alpha}
\quad \text{and} \quad
K(t) = - \alpha \log(1 - \beta t),
\end{equation*}
so 
\begin{equation}\label{eq:KappaGamma}
\kappa_r = \alpha \beta^r(r-1)! \quad \text{for $r \in \N$}.
\end{equation}
For even $d\geq 2$,
\begin{equation}\label{eq:GammaMGF}
\cnorm{A}_{\X,d}^d 
= [t^d] \prod_{i=1}^n \frac{1}{(1-\beta \lambda_i t)^{\alpha}}
=[t^d] \Bigg(\frac{1}{\beta^n t^n p_A(\beta^{-1}t^{-1})}\Bigg)^{\alpha}
\quad \text{for $A \in \h_n$},
\end{equation} 
in which $p_A(t) = \det(tI-A)$ denotes the characteristic polynomial of $A$.

\begin{example}\label{Example:Gamma4}
Since $\kappa_1 = \alpha \beta$ and $\kappa_2 = \alpha \beta^2$, \eqref{eq:cn2} becomes
\begin{equation*}
\cnorm{Z}_{\X,2}^2
= \frac{1}{2} \alpha  \beta^2 \tr (Z^*Z) + \frac{1}{2} \alpha^2 \beta^2 (\tr Z^*)(\tr Z)
\quad \text{for $Z\in \M_n$}.
\end{equation*}
Similarly, \eqref{eq:Z44} yields
\begin{align*}
\cnorm{Z}_{\X,4}^4
&= \frac{1}{24}\big( 
\alpha^4 \beta^4 (\tr Z)^2 (\tr Z^*)^2
+\alpha^3 \beta^4 (\tr Z^*)^2 \tr(Z^2) 
\\
&\qquad +4 \alpha^3 \beta^4 (\tr Z)(\tr Z^*)(\tr Z^* Z)
+2 \alpha^2 \beta^4 (\tr Z^* Z)^2 \\
&\qquad
+\alpha^3 \beta^4 (\tr Z)^2 \tr(Z^{*2})
+\alpha^2 \beta^4 \tr (Z^2) \tr (Z^{*2})
\\
&\qquad
+4 \alpha^2 \beta^4 \tr(Z^*)\tr (Z^* Z^2)
+4 \alpha^2 \beta^4 \tr (Z) \tr (Z^{*2} Z)
\\
&\qquad
+2 \alpha  \beta^4 \tr (Z^* Z Z^* Z)
+4 \alpha  \beta^4 \tr (Z^{*2} Z^2)
\big).
\end{align*}
These formulae generalize \cite[eq.~8 \& 9]{Aguilar}, which correspond to $\alpha = \beta = 1$.
\end{example}

%%%%%%%%%%%%%%%%%%%%%%%%%%%%%%%%%%
\subsection{A generalization of Hunter's positivity theorem}\label{Subsection:Hunter}
Examining a special case of the gamma distribution (Subsection \ref{Subsection:Gamma}) 
recovers Hunter's theorem \cite{Hunter} (Corollary \ref{Corollary:Hunter} below)
and also establishes a powerful generalization (Theorem \ref{Theorem:GeneralHunter}).

\begin{example}\label{Example:Exponential}
Let $\alpha=\beta=1$ in \eqref{eq:Gamma} and \eqref{eq:GammaMGF}.  Then
\begin{equation}\label{eq:CHS1}
\cnorm{A}_{\X,d}^2
= [t^d] \prod_{i=1}^n\frac{1}{1-\lambda_it}=[t^d] \frac{1}{t^np_A(t^{-1})}
\quad \text{for $A \in \h_n$},
\end{equation} 
which is \cite[Thm.~20]{Aguilar}.
Expand each factor $(1 - \lambda_i t)^{-1}$
as a geometric series, multiply out the result, and deduce that for $d \geq 2$ even,
\begin{equation}\label{eq:CHS2}
\cnorm{A}_{\X,d}^d 
=[t^d] \prod_{i=1}^n\frac{1}{1-\lambda_it}
=[t^d] \sum_{r=0}^{\infty} h_r(\lambda_1, \lambda_2, \ldots, \lambda_n) t^r.
\end{equation}
From \eqref{eq:KappaGamma}, we have $\kappa_i=(i-1)!$. Therefore,
\begin{equation*}
\frac{\kappa_{\vec{\pi}}}{y_{\vec{\pi}}}=\frac{\prod_{i\geq 1} \big[(i-1)!\big]^{m_i}}{\prod_{i\geq 1} (i!)^{m_i} m_i!}=\frac{1}{\prod_{i\geq 1}i^{m_i}m_i!}
\end{equation*} for any partition $\vec{\pi}$. Theorem \ref{Theorem:Main} and \eqref{eq:DefY} imply that for even $d\geq 2$ and $A \in \h_n$,
\begin{equation}\label{eq:StanleyPowerSum}
h_d(\lambda_1, \lambda_2, \ldots, \lambda_n)
= \cnorm{A}_{\X,d}^d=\sum_{\vec{\pi}\vdash d}\frac{\kappa_{\vec{\pi}}p_{\vec{\pi}}}{y_{\vec{\pi}}}
=\sum_{\vec{\pi}\vdash d}\frac{p_{\vec{\pi}}}{z_{\vec{\pi}}} ,
\end{equation} 
in which $z_{\vec{\pi}}=\prod_{i\geq 1}i^{m_i}m_i!$ and $p_{\vec{\pi}}$ is given by \eqref{eq:pTrace}.
This recovers the combinatorial representation of even-degree CHS polynomials \cite[Prop.~7.7.6]{StanleyBook2} and establishes Hunter's positivity theorem since $\cnorm{\cdot}_{\X,d}^d$ is positive definite.
\end{example}

The next theorem generalizes Hunter's theorem \cite{Hunter}, which is the case $\alpha = 1$.

\begin{theorem}\label{Theorem:GeneralHunter}
For even $d\geq 2$ and $\alpha \in \N$, 
\begin{equation*}
H_{d,\alpha}(x_1,x_2, \ldots, x_n)= \sum_{\substack{\vec{\pi}\vdash d\\ |\vec{\pi}|\leq\alpha}} 
c_{\vec{\pi}} h_{\vec{\pi}}(x_1,x_2,\ldots,x_n)
\end{equation*}
is positive definite on $\R^n$, in which 
the sum runs over all partitions $\vec{\pi}=(\pi_1,\pi_2,\ldots,\pi_r)$ of $d$.
Here, $h_{\vec{\pi}}=h_{\pi_1}h_{\pi_2}\cdots h_{\pi_r}$ is a product of complete homogeneous symmetric 
polynomials and
\begin{equation*}
c_{\vec{\pi}} = \frac{\alpha !}{ (\alpha-|\vec{\pi}|)! \prod_{i=1}^r m_i!},
\end{equation*}
where $|\vec{\pi}|$ denotes the number of parts in $\vec{\pi}$
and $m_i$ is the multiplicity of $i$ in $\vec{\pi}$.
\end{theorem}

\begin{proof}
Let $\alpha\in \N$ and define polynomials $P_{\ell}^{(\alpha)}(x_1, x_2, \ldots, x_{\ell})$ by
\begin{equation}\label{eq:PPalpha}
P_0^{(\alpha)}=x_0=1
\quad \text{and} \quad 
\Big(1+ \sum_{r=1}^{\infty} x_r t^r \Big)^{\alpha}
=\sum_{\ell=0}^{\infty}P_{\ell}^{(\alpha)}(x_1, x_2, \ldots, x_{\ell})t^{\ell}.
\end{equation} 
Then
\begin{equation}\label{eq:PPCombo}
P_{\ell}^{(\alpha)}( x_1,x_2, \ldots, x_{\ell})
\,\,=\!\!
\sum_{\substack{ i_1, i_2, \ldots, i_{\alpha}\leq \ell \\ i_1+i_2+\cdots+i_{\alpha}=\ell}} x_{i_1}x_{i_2}\cdots x_{i_{\alpha}}
=\sum_{\substack{\vec{\pi}\vdash \ell \\ |\vec{\pi}|\leq\alpha}} c_{\vec{\pi}} x_{\vec{\pi}}.
\end{equation}
Let $\X$ be a random vector whose $n$ components are iid random variables
distributed according to \eqref{eq:Gamma} with $\beta = 1$.
Let $A\in \h_n$ have eigenvalues $x_1,x_2,\ldots,x_n$.
Then for even $d \geq 2$,
\begin{align*}
\cnorm{A}_{\X,d}^d
&\overset{\eqref{eq:GammaMGF}}{=} [t^d] \bigg( \prod_{i=1}^k \frac{1}{1-x_i t} \bigg)^{\alpha} \\
&\overset{\eqref{eq:CHS2}}{=} [t^d] \bigg(1+\sum_{r=1}^{\infty} h_{r}(x_1, x_2, \ldots, x_n) t^{r} \bigg)^{\alpha} \\
&\overset{\eqref{eq:PPalpha}}{=} [t^d] \sum_{\ell=0}^{\infty} P_{\ell}^{(\alpha)}(h_1, h_2, \ldots, h_{\ell})t^{\ell} \\
&\overset{\eqref{eq:PPCombo}}{=} [t^d]\sum_{\ell=0}^{\infty}
\bigg(
\sum_{\substack{\vec{\pi}\vdash \ell \\ |\vec{\pi}|\leq\alpha}} c_{\vec{\pi}} h_{\vec{\pi}}(x_1, x_2, \ldots, x_n) \bigg)t^{\ell}.
\end{align*}
Consequently,
$\displaystyle
\sum_{\substack{\vec{\pi}\vdash d \\ |\vec{\pi}|\leq \alpha}} c_{\vec{\pi}} h_{\vec{\pi}}(x_1, x_2, \ldots, x_n) 
 = \cnorm{A}_{\X,d}^d$,
which is positive definite.
\end{proof}

As mentioned in Example \ref{Example:Exponential}, $\alpha = 1$ recovers
Hunter's theorem.

\begin{corollary}[Hunter \cite{Hunter}]\label{Corollary:Hunter}
For even $d \geq 2$, the complete symmetric homogeneous polynomial 
$h_d(x_1,x_2,\ldots,x_n)$ is positive definite.
\end{corollary}

\begin{example} If $\alpha = 2$, then we obtain the positive definite symmetric polynomial
\begin{equation*}
H_{d,2}(x_1,x_2, \ldots, x_n)= \sum_{i=0}^d h_i (x_1,x_2, \ldots, x_n) h_{d-i}(x_1,x_2, \ldots, x_n).
\end{equation*}
\end{example}

\begin{example}
The relation
\begin{equation*}
\sum_{\ell=0}^{\infty} H_{\ell, \alpha}t^{\ell}=\Big(\sum_{\ell=0}^{\infty} h_{\ell}t^{\ell}\Big)\Big(\sum_{\ell=0}^{\infty} H_{\ell, \alpha-1}t^{\ell}\Big)
\end{equation*} implies that the sequence $\{H_{d,\alpha}\}_{\alpha\geq 1}$ satisfies the recursion
\begin{equation}
H_{d,\alpha}=\sum_{i=0}^d h_i H_{d-i, \alpha-1}.\label{eq:HRecursion}
\end{equation}
For example, let $j=4$ and $\alpha=3$. There are four partitions $\vec{\pi}$ of $j$ with $|\vec{\pi}|\leq 3$. These are $(1,1,2)$,  $(1,3)$, $(2,2)$ and $(4)$. 
Therefore,
\begin{align*}
H_{4,3}(x_1, x_2, x_3,x_4)&=c(1,1,2)h_1^2h_2+c(1,3)h_1h_3+c(2,2)h_2^2+c(4)h_4\\
&=\frac{3!}{0! 2! 1!}h_1^2h_2+\frac{3!}{1! 1! 1!}h_1h_3+\frac{3!}{1! 2!}h_2^2+\frac{3!}{2!1!}h_4\\[2pt]
&=3h_1^2h_2+6h_1h_3+3h_2^2+3h_4
\end{align*} 
is a positive definite symmetric polynomial. In light of \eqref{eq:HRecursion}, we can also write
\begin{equation*}
H_{4,3}(x_1, x_2, x_3,x_4)=\sum_{i=0}^4 h_i H_{4-i, 2}=H_{4,2}+h_1H_{3,2}+h_2H_{2,2}+h_3H_{1,2}+h_4.
\end{equation*}
\end{example}

%%%%%%%%%%%%%%%%%%%%%%%%%%%%%

\subsection{Normal random variables}\label{Subsection:Normal}
    Let $\X=(X_1, X_2, \ldots, X_n)$, in which the $X_i$ are independent 
    normal random variables with mean $\mu$ and variance $\sigma^2>0$.  Then
    \begin{equation*}
      M(t)=\exp\!\Big(t\mu+\frac{\sigma^2t^2}{2} \Big)
      \quad \text{and} \quad
      K(t) = \frac{\sigma^2 t^2}{2}+\mu  t;
    \end{equation*} 
    in particular, $\kappa_1 = \mu$ and $\kappa_2 = \sigma^2$ and all higher cumulants are zero.    
    Then
    \begin{equation*}
    	M_{\X,\LL}(t) 
	= \prod_{i=1}^n \exp\!\Big(\lambda_i t\mu+\frac{\sigma^2 \lambda_i^2 t^2}{2} \Big)
	=\exp\! \Big( t\mu\tr A+\frac{\sigma^2\tr(A^2)t^2}{2} \Big).
    \end{equation*}
    Theorem \ref{Theorem:Main} and the above tell us that
    \begin{equation}\label{eq:Same}
        \cnorm{A}_{\X,d}^d=\sum_{k=0}^{\frac{d}{2}}  \frac{\mu^{2k} (\tr A)^{2k}}{(2k)!} 
        \cdot  \frac{\sigma^{d-2k} \norm{A}_{\operatorname{F}}^{d-2k}}{2^{\frac{d}{2}-k} (\frac{d}{2}-k)!}
        \quad \text{for $A \in \h_n$},
    \end{equation}      
    in which $\norm{A}_{\operatorname{F}}$ is the Frobenius norm of $A$.
    For $d\geq 2$ even, Theorem \ref{Theorem:Main} yields
    \begin{align*}
    	\cnorm{Z}_{\X,2}^2
	&= \frac{1}{2} \sigma^2 \tr(Z^*\!Z) + \frac{1}{2} \mu^2 (\tr Z^*)(\tr Z), \\
	\cnorm{Z}_{\X,4}^4
	&= \frac{1}{24} \Big(
	\mu^4 (\tr Z)^2 (\tr Z^*)^2
	+\mu^2 \sigma^2 \tr (Z^*)^2\tr (Z^2) \\
	&\qquad\qquad +4 \mu^2 \sigma^2 (\tr Z)  (\tr Z^*) (\tr Z^* Z)
	+2 \sigma^4 (\tr Z^* Z)^2
	\\
	&\qquad\qquad +\mu^2 \sigma^2 (\tr Z)^2 \tr (Z^{*2})
	+\sigma^4 \tr (Z^2) \tr (Z^{*2})
	\Big).
    \end{align*}
    Since $\kappa_r=0$ for $r \geq 3$, we see that
    $\cnorm{Z}_{\X,4}^4$ does not contain summands like
    $\tr(Z^*)\tr (Z^* Z^2)$ and $\tr (Z^{*2} Z^2)$, in contrast to the formula in Example \ref{Example:Gamma4}.

%%%%%%%%%%%%%%%%%%%%%%%%%%%%%%%%%%%%%%%%%%%
\subsection{Uniform random variables}\label{Subsection:Uniform}
Let $\X=(X_1, X_2, \ldots, X_n)$, where the $X_i$ are independent and uniformly distributed on $[a,b]$. 
Each $X_i$  has probability density $f(x)=(b-a)^{-1}\mathbbm{1}_{[a,b]}$, where $\mathbbm{1}_{[a,b]}$ is the indicator function of $[a,b]$. Then 
\begin{equation}\label{eq:MomentUniform}
\mu_k = \E  [X_i^k] = \int_{-\infty}^{\infty} x^k f(x)\,dx= \frac{h_k(a,b)}{k+1},
\end{equation} 
in which $h_k(a,b)$ is the CHS polynomial of degree $k$ in the variables $a,b$. The moment and cumulant generating functions of each $X_i$ are
\begin{equation*}
M(t)=\frac{e^{bt}-e^{at}}{t(b-a)}
\quad \text{and} \quad
K(t) = \log \bigg(\frac{e^{t (b-a)}-1}{t (b-a)}\bigg)+a t.
\end{equation*}
The cumulants are
\begin{equation*}
\kappa_r =  
\begin{cases}
\dfrac{a+b}{2} & \text{if $r=1$},\\[8pt]
\dfrac{B_r}{r}(b-a)^r & \text{if $r$ is even},\\[8pt]
0 & \text{otherwise},
\end{cases}
\end{equation*}
in which $B_r$ is the $r$th Bernoulli number \cite{Gould}. 
Theorem \ref{Theorem:Main} ensures that 
\begin{equation}\label{eq:GeneralSinh}
\cnorm{A}_{\X,d}^d = [t^d]\prod_{i=1}^n\frac{ e^{b\lambda_i t}-e^{a\lambda_i t}}{\lambda_it(b-a)}
\quad \text{for $A \in \h_n$}.
\end{equation}

\begin{example}\label{Example:UniformSymmetric}
If $[a,b]=[-1,1]$, then 
\begin{equation*}
\cnorm{Z}_{\X,4}^4= \frac{1}{1080}\big(
10 (\tr Z^*Z)^2+5 \tr (Z^2) \tr (Z^{*2})-4 (\tr Z^2Z^{*2})-2 \tr (ZZ^*ZZ^*)\big)
\end{equation*}
for $Z \in \M_n$, which is not obviously positive, let alone a norm.  
Indeed, $\tr Z^2Z^{*2}$ and $\tr (ZZ^*ZZ^*)$ appear with negative scalars in front of them!
Similarly,
\begin{equation*}
\cnorm{A}_{\X,6}^6
= \frac{1}{45360}
\big(35 (\tr A^2)^3-42 \tr(A^4) \tr (A^2)+16 \tr(A^6) \big)
\quad \text{for $A \in \h_6$}
\end{equation*}
has a nonpositive summand.
Observe that
\begin{equation*}
M_{\X,\LL}(t)= \prod_{i=1}^n\frac{\sinh(\lambda_it)}{\lambda_i t} 
\end{equation*}
is an even function of each $\lambda_i$, so the corresponding norms are polynomial functions
in even powers of the eigenvalues (so positive definiteness is no surprise, although the triangle inequality 
for these norms remains nontrivial).  
\end{example}

\begin{example}
If $[a,b]=[0,1]$, then
\begin{equation*}
M_{\X,\LL}(t) = \prod_{i=1}^n \frac{e^{\lambda_i t}-1}{\lambda_i t} ,
\end{equation*}
and hence for $A \in \h_n$,
\begin{align*}
\cnorm{A}_{\X,2}^2  &= \tfrac{1}{12}(2 \lambda_1^2+3 \lambda_1 \lambda_2+2 \lambda_2^2), \\
\cnorm{A}_{\X,4}^4  &= \tfrac{1}{720}( 6 \lambda_1^4
+15 \lambda_1^3 \lambda_2
+20 \lambda_1^2 \lambda_2^2 
+15 \lambda_1 \lambda_2^3
+6 \lambda_2^4  
).
\end{align*}
Unlike the previous example, these symmetric polynomials are not obviously positive definite
since $\lambda_1^3 \lambda_2$ and $\lambda_1 \lambda_2^3$ need not be nonnegative.
\end{example}

%%%%%%%%%%%%%%%%%%%%%%%%%%%%%%%%%%%%%%%%%%%%%%

\subsection{Laplace random variables}\label{Subsection:Laplace}
Let $\X=(X_1, X_2, \ldots, X_n)$, where 
the $X_i$ are independent random variables distributed according to the probability density
\begin{equation*}
f(x)=\frac{1}{2\beta} e^{-\frac{|x-\mu|}{\beta}},
\quad \text{in which $\mu \in \R$ and $\beta>0$}.
\end{equation*} 
The moment and cumulant generating functions of the $X_i$ are
\begin{equation*}
M(t)=\frac{e^{\mu t}}{1-\beta^2t^2}
\quad \text{and} \quad
K(t) = \mu  t-\log (1-\beta^2 t^2),
\end{equation*} 
respectively.
The cumulants are
\begin{equation*}
\kappa_r = 
\begin{cases}
\mu & \text{if $r=1$},\\[2pt]
2 \beta^r (r-1)! & \text{if $r$ is even},\\[2pt]
0 & \text{otherwise}.
\end{cases}
\end{equation*}
For even $d\geq 2$, it follows that $\cnorm{A}_{\X,d}^d$ is the $d$th term in the Taylor expansion of 
\begin{equation}\label{eq:Laplace}
\cnorm{A}_{\X,d}^d 
= [t^d] \prod_{i=1}^n\frac{ e^{\mu t}}{1-\beta^2\lambda_i^2t^2}.
=  e^{\mu \tr A t} [t^d]\prod_{i=1}^n\frac{ 1}{1-\beta^2\lambda_i^2t^2}.
\end{equation} 

\begin{example}
Let $\mu=\beta=1$. Expanding the terms in \eqref{eq:Laplace} gives
\begin{equation*}
M_{\X,\LL}(t)=e^{ \tr A t}\prod_{i=1}^n\frac{ 1}{1-\lambda_i^2t^2}=\Big(\sum_{k=0}^{\infty} (\tr A)^k\frac{t^k}{k!}  \Big)\Big(\sum_{k=0}^{\infty} h_k(\lambda_1^2, \lambda_2^2, \ldots, \lambda_n^2)t^{2k}    \Big),
\end{equation*}
which implies
\begin{equation*}
\cnorm{A}_{\X,d}^d=\sum_{k=0}^{d/2} \frac{(\tr A)^{2k}}{(2k)!} h_{\frac{d}{2}-k}(\lambda_1^2, \lambda_2^2, \ldots, \lambda_n^2).
\end{equation*} 
\end{example}

%%%%%%%%%%%%%%%%%%%%%%%%%%%%%%%%%%%%%%%%%%%

\subsection{Bernoulli random variables}\label{Subsection:Bernoulli}
Let $\X=(X_1, X_2, \ldots, X_n)$,  in which the $X_i$ are independent Bernoulli random variables. Each $X_i$ takes values in $\{0,1\}$ 
with $\P(X_i=1)=q$ and $\P(X_i=0)=1-q$ for some fixed $0<q<1$. Each $X_i$ satisfies
\begin{equation*}
\E [X_i^k]=\sum_{j\in \{0,1\}} j^k\P(X_i=j)=q \quad \text{for $k \in \N$}.
\end{equation*} We have
\begin{equation*}
M(t) = 1-q + qe^t 
\quad \text{and} \quad
K(t) = \log(1-q+qe^t).
\end{equation*} The first few cumulants are 
\begin{equation*}
q, \qquad q-q^2, \qquad 2 q^3-3 q^2+q, \qquad -6 q^4+12 q^3-7 q^2+q,\ldots.
\end{equation*}

For even $d\geq 2$, the multinomial theorem and independence imply that
\begin{equation*}
\cnorm{A}_{\X,d}^d
=\frac{1}{d!} \sum_{i_1+ i_2+\cdots + i_n=d} q^{|I|} \lambda_1^{i_1}\lambda_2^{i_2}\cdots \lambda_n^{i_n},
\end{equation*} 
in which $|I|$ denotes the cardinality of $I=\{i_1, i_2, \ldots, i_n\}$. 
We can write this as
\begin{equation*}
\cnorm{A}_{\X,d}^d=\sum_{\vec{\pi}\, \vdash \, d} \frac{|\vec{\pi}|!}{d!}q^{| \vec{\pi}|} m_{\vec{\pi}}(\vec{\lambda}),
\end{equation*} 
in which $m_{\vec{\pi}}$ denotes the \emph{monomial symmetric polynomial} 
corresponding to the partition $\vec{\pi}$ of $d$ \cite[Sect.~7.3]{StanleyBook2}.
To be more specific,
\begin{equation*}
m_{\vec{\pi}}(\vec{x})=\sum_{\vec{\alpha}} x^{\vec{\alpha}},
\end{equation*} 
in which the sum is taken over all distinct permutations $\vec{\alpha}=(\alpha_1, \alpha_2, \ldots, \alpha_r)$ of the entries of $\vec{\pi}=(i_1, i_2, \ldots, i_r)$ and $x^{\vec{\alpha}}=x_1^{\alpha_1}x_2^{\alpha_2}\cdots x_r^{\alpha_r}$. For example,
\begin{equation*}
m_{(1)} =\sum_i x_i, \qquad 
m_{(2)} =\sum_i x_i^2,\quad \text{and} \quad
m_{(1,1)}=\sum_{i<j}x_ix_j.
\end{equation*}

%%%%%%%%%%%%%%%%%%%%%%%%%%%%%%%%%%%%%%%%%%%
\subsection{Finite discrete random variables}\label{Subsection:Finite}

Let $X$ be supported on $\{a_1, a_2, \ldots, a_{\ell}\} \subset \R$, with
$\P(X=a_j)=q_j>0$ for $1\leq j \leq \ell$ and $q_1+q_2+\cdots +q_{\ell}=1$. Then
\begin{equation*}
\E[X^k]=\sum_{i=1}^{\ell} a_i^k q_i,
\end{equation*} 
and hence
\begin{equation}
M(t)=\sum_{j=1}^{\ell}q_j\bigg(\sum_{k=0}^{\infty}  a_j^k \frac{t^k}{k!}\bigg)=\sum_{j=1}^{\ell} q_j e^{a_jt}.\label{eq:FiniteDiscrete}
\end{equation} 
Let $\X=(X_1, X_2, \ldots, X_n)$, in which $X_1, X_2, \ldots, X_n\sim X$ are iid random variables.

\begin{example}\label{Example:Rademacher}
Let $\ell=2$ and $a_1=-a_2=1$ with $q_1=q_2=\frac{1}{2}$. 
The $X_i$ are \emph{Rademacher} random variables. Identity 
\eqref{eq:FiniteDiscrete} implies that $M(t)=\cosh t$, so
    \begin{equation*}
    	M_{\X,\LL}(t) = \prod_{i=1}^n\cosh(\lambda_it).
    \end{equation*}
    For $n=2$, 
    \begin{align*}
    \cnorm{A}_{\X,2}^2  &=  \tfrac{1}{2} (\lambda_1^2+\lambda_2^2), \\
\cnorm{A}_{\X,4}^4  &= \tfrac{1}{24} (\lambda_1^4+6 \lambda_2^2 \lambda_1^2+\lambda_2^4), \quad \text{and}\\
\cnorm{A}_{\X,6}^6  &=\tfrac{1}{720} (\lambda_1^6+15 \lambda_2^2 \lambda_1^4+15 \lambda_2^4 \lambda_1^2+\lambda_2^6).
    \end{align*}
\end{example}

Let $\gamma_p=\sqrt{2} (\sqrt{\pi})^{-1/p }\Gamma(\frac{p+1}{2})^{1/p}$ 
denote the $p$th moment of a standard normal random variable.  Let $X_1, X_2, \ldots, X_n$ be independent Rademacher random variables (see Example \ref{Example:Rademacher}). The classic Khintchine inequality asserts that 
\begin{equation}\label{eq:Classic}
\Big(\E \Big|\sum_{i=1}^n \lambda_iX_i \Big|^2\Big)^{1/2}\leq \Big(\E \Big|\sum_{i=1}^n \lambda_iX_i\Big|^p\Big)^{1/p}\leq a_p\Big(\E \Big|\sum_{i=1}^n \lambda_iX_i\Big|^2\Big)^{1/2}
\end{equation} 
for all $\lambda_1, \lambda_2, \ldots, \lambda_n\in \R$ and $p\geq 2$, with $a_2=1$  and $a_p=\gamma_p$ for $p>2$. Moreover, these constants are optimal \cite{Haagerup}. Immediately, we obtain the equivalence of norms
\begin{equation}\label{eq:NormEquivalence}
\norm{A}_{\mathrm{F}}\leq \Gamma(p+1)^{1/p}\cnorm{A}_{\X,p}\leq a_p\norm{A}_{\mathrm{F}},
\end{equation} 
for all $A\in \h_n(\mathbb{C})$ and $p\geq 2$. The proof of Theorem 1.e  implies that 
\begin{equation*}
\norm{Z}_{\mathrm{F}}\leq \Gamma(p+1)^{1/p}\cnorm{Z}_{\X,p}\leq a_p\norm{Z}_{\mathrm{F}},
\end{equation*} 
for all $Z\in \M_n$ and $p\geq 2$.

In general, suppose that $X_1, X_2, \ldots, X_n$ are iid random variables. 
A comparison of the form \eqref{eq:Classic} is a \emph{Khintchine-type inequality}. 
Establishing a Khintchine-type inequality here is equivalent to establishing an equivalence of norms as in \eqref{eq:NormEquivalence}. This is always possible since $\h_n(\C)$ is finite dimensional. 
However, establishing Khintchine-type inequalities is, in general, a nontrivial task;
this is an active research area in probability theory \cite{ENT1, ENT2, Havrilla, LO}.

%%%%%%%%%%%%%%%%%%%%%%%%%%%%
\subsection{Poisson random variables}\label{Subsection:Poisson}
Let $\X=(X_1, X_2, \ldots, X_n)$, in which the $X_i$ are independent random variables on 
$\{0,1,2,\ldots\}$ distributed according to 
\begin{equation*}
f(t)= \frac{e^{-\alpha} \alpha^t}{t!},
\quad \text{in which $\alpha >0$}.
\end{equation*}
The moment and cumulant generating functions of the $X_i$ are
\begin{equation*}
M(t)= e^{\alpha  (e^t-1)}
\quad \text{and} \quad
K(t) = \alpha  (e^t-1),
\end{equation*} 
respectively.  Therefore, $\kappa_i = \alpha$ for all $i \in \N$ and hence
\begin{equation*}
\cnorm{A}_{\X, d}^d=\sum_{\vec{\pi}\vdash d} \frac{   \alpha^{|\vec{\pi}|} p_{\vec{\pi}}}{y_{\vec{\pi}}},
\end{equation*} 
For example, if $A \in \h_n$ we have
\begin{align*}
4!\cnorm{A}_{\X,4}^4
=\alpha^4(\tr A)^4+6\alpha^3(\tr A)^2\tr A^2+4\alpha^2\tr A\tr A^3+3\alpha^2(\tr A^2)^2+\alpha \tr A^4.
\end{align*}

%%%%%%%%%%%%%%%%%%%%%%%%%%%%
\subsection{Pareto random variables}\label{Subsection:Pareto}
Let $\X=(X_1, X_2, \ldots, X_n)$, in which the $X_i$ are independent random variables distributed according 
to the probability density
\begin{equation*}
f(x) = 
\begin{cases}
\dfrac{\alpha}{x^{\alpha+1}} & x\geq 1, \\[5pt]
0 & x < 1.
\end{cases}
\end{equation*}
Such random variables do not admit a moment generating function.
The moments that do exist are 
\begin{equation*}
\mu_k = \frac{\alpha}{\alpha - k} \quad \text{for $k < \alpha$}.
\end{equation*}
For even $d\geq2$ with $d < \alpha$, the multinomial theorem and independence yield
\begin{align*}
d!\cnorm{A}_{\X ,d}^d 
&= \E[ \langle \X , \vec{\lambda}\rangle^d ] 
= \E\big[(\lambda_1 X_1 + \lambda_2 X_2 +  \cdots + \lambda_n X_n)^d \big] \\
&= \E \bigg[ \sum_{ \substack{k_1+k_2+\cdots+k_n=d \\ k_1,k_2,\ldots,k_n \geq 0}} 
\binom{d}{k_1,k_2,\ldots,k_n} \prod_{i=1}^n (\lambda_i X_i)^{k_i} \bigg] \\
&=  \sum_{ \substack{k_1+k_2+\cdots+k_n=d \\ k_1,k_2,\ldots,k_n \geq 0}} 
\binom{d}{k_1,k_2,\ldots,k_n} \prod_{i=1}^n \E\big[(\lambda_i X_i)^{k_i}\big] \\
&=  \sum_{ \substack{k_1+k_2+\cdots+k_n=d \\ k_1,k_2,\ldots,k_n \geq 0}} 
\binom{d}{k_1,k_2,\ldots,k_n} \prod_{i=1}^n \lambda_i^{k_i} \E\big[X_i^{k_i}\big] \\
&=  \sum_{ \substack{k_1+k_2+\cdots+k_n=d \\ k_1,k_2,\ldots,k_n \geq 0}} 
\binom{d}{k_1,k_2,\ldots,k_n} \prod_{i=1}^n \frac{\alpha \lambda_i^{k_i}}{\alpha - k_i} .
\end{align*}
In particular,
$\lim_{\alpha \to\infty}d!\cnorm{A}_{\X_{\alpha},d}^d = (\tr A)^d$
and
\begin{align*}
\lim_{\alpha \to d^{+}}
(\alpha-d)d!\cnorm{A}_{\X_{\alpha},d}^d
&= \lim_{\alpha \to d^{+}} (\alpha-d) \sum_{ \substack{k_1+k_2+\cdots+k_n=d \\ k_1,k_2,\ldots,k_n \geq 0}} 
\binom{d}{k_1,k_2,\ldots,k_n} \prod_{i=1}^n \frac{\alpha \lambda_i^{k_i}}{\alpha - k_i}  \\
&= \lim_{\alpha \to d^{+}}(\alpha-d)
\sum_{i=1}^n \binom{d}{d}\frac{d\lambda_i^{d}}{\alpha-d} = d \sum_{i=1}^n \lambda_i^{d}
= d\norm{A}_d^d,
\end{align*}
in which $\norm{A}_d$ is the Schatten $d$-norm on $\h_n$.

\begin{example}
For $n=2$, 
\begin{align*}
\cnorm{A}_{\X ,2}^2 &=\frac{1}{2}\alpha\left(\frac{\lambda_1^2}{\alpha-2}+\frac{2\alpha\lambda_1\lambda_2}{(\alpha-1)^2}+\frac{\lambda_2^2}{\alpha-2} \right), \quad \text{and}\\
\cnorm{A}_{\X ,4}^4 &= \frac{1}{24}\alpha \left(\frac{\lambda_1^4}{\alpha-4} + \frac{4\alpha\lambda_1^3\lambda_2}{\alpha^2-4\alpha+3} + \frac{6\alpha\lambda_2^2\lambda_1^2}{(\alpha-2)^2} + \frac{4\alpha\lambda_1\lambda_2^3}{\alpha^2-4\alpha+3} +
\frac{\lambda_2^4}{\alpha-4}\right).
\end{align*}
\end{example}

%%%%%%%%%%%%%%%%%%%%%%%%%%%%%%%%%%%%%%%%%%%%%%%%%%%%%%%%%%
\section{Proof of Theorem \ref{Theorem:Main}}\label{Section:Proof}
Let $d\geq 2$ be arbitrary and let $\X=(X_1, X_2, \ldots, X_n)$ be a random vector in $\R^n$, in which $X_1, X_2, \ldots, X_n\in L^d(\Omega, \F, \P)$ are independent and identically distributed (iid) random variables.  Independence is not needed for (a); see Remark \ref{Remark:Conjugation}.
We let $\LL=(\lambda_1, \lambda_2, \ldots, \lambda_n)$ denote the vector of eigenvalues of $A \in \h_n$.  As before, $A$ denotes a typical
Hermitian matrix and $Z \in \M_n$ an arbitrary square matrix.

The proofs of (a) - (e) of Theorem \ref{Theorem:Main} are placed in separate subsections below. Before we proceed, we require an important lemma.  

\begin{lemma}\label{Lemma:PosDef}
$\X$ is positive definite.
\end{lemma}

\begin{proof}
H\"older's inequality shows that each $X_i \in L^2(\Omega, \F, \P)$, so $\mu_1$ and $\mu_2$ are finite. 
Jensen's inequality yields $\mu_1^2\leq \mu_2$; nondegeneracy of the $X_i$ ensures the inequality is strict. 
Independence implies that $\E[X_i X_j] = \E[X_i]\E[X_j]$ for $i\neq j$, so
\begin{equation*}
\Sigma(\X) = [\E X_iX_j ]=
\begin{bmatrix} 
\mu_2 & \mu_1^2 & \cdots & \mu_1^2 \\[2pt] 
\mu_1^2 & \mu_2 & \cdots & \mu_1^2 \\
\vdots & \vdots & \ddots & \vdots \\
\mu_1^2 & \mu_1^2 & \cdots & \mu_2 \\
\end{bmatrix}
=(\mu_2-\mu_1^2)I+\mu_1^2J,
\end{equation*}
in which $\mu_2-\mu_1^2>0$ and $J$ is the all-ones matrix. Thus,
$\Sigma(\X)$ is the sum of a positive definite and a positive semidefinite matrix, so it is positive definite. 
\end{proof}

%%%%%%%%%%%%%%%%%%%%%%%%%%%%%%%%%%%%%%
\subsection{Proof of Theorem \ref{Theorem:Main}.a}
Since $X_1, X_2, \ldots, X_n\in L^d(\Omega, \F, \P)$ for some $d\geq 2$,
H\"older's inequality implies the random variable $\Lambda=\langle \X, \LL\rangle$ satisfies 
\begin{equation}\label{Eq:PosDef}
    \langle \LL, \Sigma(\X)\LL\rangle = \E [| \Lambda |^2] \leq (\E | \Lambda |^d)^{2/d}.
\end{equation}

For $A\in \h_n$, consider the nonnegative function
\begin{equation}\label{eq:RealNorm}
\NN(A)=\bigg(\,\frac{  \E |\langle \X, \LL\rangle|^d}{\Gamma(d+1)}\, \bigg)^{1/d}.
\end{equation} 
It is clearly homogeneous: $\NN(\alpha A)=|\alpha| \NN(A)$ for all $\alpha \in \R$. 
Lemma \ref{Lemma:PosDef} ensures that $\Sigma(\X)$ is positive definite, so \eqref{Eq:PosDef} implies $\NN(A) = 0$ 
if and only if $A = 0$.

We must show
that $\NN$ satisfies the triangle inequality. 
Our approach parallels that of \cite[Thm.~1]{Aguilar}:
for $d\geq 2$ even, $\mathfrak{H}(A) = h_d(\lambda_1(A),\lambda_2(A),\ldots, \lambda_n(A))^{1/d}$ is a norm on $\h_n$.
We first show that $\NN$ satisfies the triangle inequality on $\mathrm{D}_n(\R)$, the space of real diagonal matrices. 
Then we use Lewis' framework for convex matrix analysis \cite{LewisGroup} 
to establish the triangle inequality on $\h_n$. 

Let $\V$ be a finite-dimensional real vector space with inner product $\langle \cdot, \cdot \rangle_{\V}$. 
The adjoint $\phi^*$ of a linear map $\phi: \V\to \V$ satisfies $\langle\phi^*(A), B \rangle= \langle A, \phi(B) \rangle$ for all $A,B \in \V$. We say that $\phi$ is \emph{orthogonal} if $\phi^*\circ \phi$ is the identity.
Let $\OO(\V)$ denote the set of orthogonal linear maps on $\V$.  If  $\mathcal{G} \subset \OO(\V)$ is a subgroup, then
$f: \V\to \R$ is \emph{$\mathcal{G}$-invariant} if $f( \phi(A))=f(A)$ for all $\phi\in \mathcal{G}$ and $A\in V$.

\begin{definition}[Def.~2.1 of \cite{LewisGroup}]\label{Definition:NDS}
$\delta: \V\to \V$ is a \emph{$G$-invariant normal form} if 
\begin{enumerate}
\item $\delta$ is $\mathcal{G}$-invariant,
\item For each $A\in \V$, there is an $\phi\in \OO(\V)$  such that $A=\phi( \delta(A))$, and
\item $\langle A, B\rangle_{\V} \leq \langle \delta(A), \delta(B) \rangle_{\V}$ for all $A,B\in \V$.
\end{enumerate}
\end{definition} 

Such a triple $(\V, G, \delta)$ is a \emph{normal decomposition system} (NDS).
Let  $(\V, \mathcal{G}, \delta)$ be an NDS
and $\W \subseteq \V$ a subspace.  The \emph{stabilizer} of $\W$ in $\mathcal{G}$ is 
$\mathcal{G}_{\W} = \{ \phi \in \mathcal{G} : \phi(\W)=\W\}$.  We restrict the domain of $\phi \in \mathcal{G}_{\W}$
and consider $\mathcal{G}_{\W}$ as a subset of $\OO(\W)$.  

\begin{lemma}[Thm.~4.3 of \cite{LewisGroup}]\label{Lemma:Lewis}
Let $(\V, \mathcal{G}, \delta)$ and $(\W, \mathcal{G}_{\W}, \delta|_{\W})$ be normal decomposition systems with 
$\ran \delta \subset \W$. Then a $\mathcal{G}$-invariant function $f:\V\to \R$ is convex if and only if its restriction to $\W$ is convex. 
\end{lemma}

Let $\V=\h_n$ be the $\R$-vector space of complex Hermitian $(A=A^*$) matrices equipped with the Frobenius inner product $(A,B) \mapsto \tr AB$.
Let $\operatorname{U}_n$ denote the group of $n \times n$ complex unitary matrices.  
For $U \in \operatorname{U}_n$, define $\phi_U: \V\to \V$ by $\phi_U(A)=UAU^*$.
Then $\mathcal{G}=\{\phi_U : U\in \operatorname{U}_n\}$
is a group under composition.  We may regard it is a subgroup of $\OO(\V)$ 
since $\phi_U^*=\phi_{U^*}$.

Let $\W=\mathrm{D}_n(\R) \subset \V$ denote the set of real diagonal matrices. 
Then $\mathcal{G}_{\W} = \{ \phi_P : P \in \mathcal{P}_n\}$,
in which $\mathcal{P}_n$ is the group of $n \times n$ permutation matrices.
Define $\delta: \V\to \V$ by 
$\delta(A)=\diag (\lambda_1(A), \lambda_2(A), \ldots, \lambda_n(A))$,
 the $n \times n$ diagonal matrix with $\lambda_1(A), \lambda_2(A), \ldots, \lambda_n(A)$ on its diagonal.
Observe that $\ran \delta \subset \W$ since the eigenvalues of a Hermitian matrix are real. We maintain this notation below.

\begin{lemma}\label{Lemma:Normal}
$(\V, \mathcal{G},\delta)$ and $(\W, \mathcal{G}_{\W}, \delta|_{\W})$ are normal decomposition systems.
\end{lemma}

\begin{proof}
We claim that $(\V, \mathcal{G},\delta)$ is an NDS.
(a) Eigenvalues are similarity invariant, so $\delta$ is $\mathcal{G}$-invariant.
(b) For $A \in \V$, the spectral theorem gives a $U \in \operatorname{U}_n$ such that
$A = U\delta(A)U^* = \phi_U( \delta(A) )$.
(c) For $A,B \in \V$, note that $\tr AB\leq \tr \delta(A)\delta(B)$ \cite[Thm.~2.2]{LewisHermitian};
see \cite[Remark 10]{Aguilar} for further references.

We claim that $(\W,\mathcal{G}_{\W}, \delta|_{\W})$ is an NDS.
(a) $\delta|_{\W}$ is $\mathcal{G}_{\W}$-invariant since 
$\delta(\phi_P(A))=\delta(PAP^*) = \delta(A)$ for all $A \in \W$ and $P \in \mathcal{P}_n$.
(b) If $A \in \W$, then there is a $P \in \mathcal{P}_n$ such that $A = P\delta(A)P^* = \phi_P(\delta(A))$.  
(c) The diagonal elements of a diagonal matrix are its eigenvalues.  Thus, this property is inherited from $\V$.
\end{proof}

The function $\NN: \V\to \R$ is $\mathcal{G}$-invariant
since it is a symmetric function of $\lambda_1(A), \lambda_2(A), \ldots, \lambda_n(A)$; see Remark \ref{Remark:Conjugation}.
If $A,B\in \W$, define random variables $X=\langle \X, \LL(A)\rangle$ and $Y=\langle \X, \LL(B)\rangle$. 
Since $A$ and $B$ are diagonal, $\LL(A+B)=\LL(A)+\LL(B)$ and hence
\begin{equation*}
\big(\E \big|\langle \X,\LL(A+B)\rangle\big|^d\big)^{1/d}
=\big(\E |X+Y|^d\big)^{1/d}
\leq \big(\E |X|^d\big)^{1/d}+\big(\E |Y|^d\big)^{1/d}
\end{equation*} 
by Minkowski's inequality for $L^d(\Omega, \F, \P)$. Thus,
$\NN(A+B) \leq \NN(A) + \NN(B)$
for all $A,B\in \W$, and hence $\NN$ is convex on $\W$.
Lemma \ref{Lemma:Lewis} implies that $\NN$ is convex on $\V$.  Thus,
$\tfrac{1}{2} \NN(A+B)=\NN( \tfrac{1}{2}A+\tfrac{1}{2}B)\leq \tfrac{1}{2}\NN(A) +\tfrac{1}{2}\NN(B)$
for all $A,B\in \V$, so \eqref{eq:RealNorm} defines a norm on $\V=\h_n$. \qed

\begin{remark}\label{Remark:Conjugation}
Independence is not used in the proof of (a). Our proof only requires that the function $\cnorm{A}_{\X,d}$ be invariant with respect to unitary conjugation. If the $X_i$ are assumed to be identically distributed, but not necessarily independent, then $\cnorm{A}_{\X,d}$ is a homogeneous symmetric function of the eigenvalues of $A$. Any such function is invariant with respect to unitary conjugation.
\end{remark}

%%%%%%%%%%%%%%%%%%%%%%%%%%%%%%%%%%%%%%
\subsection{Proof of Theorem \ref{Theorem:Main}.b}
Let $d \geq 2$ be even and let $\X=(X_1, X_2, \ldots, X_n)$ be a random vector, in which $X_1, X_2, \ldots, X_n$ are 
iid random variables which admit a moment generating function $M(t)$.     Let $A \in \h_n$ have 
eigenvalues $\lambda_1 \geq \lambda_2 \geq \cdots \geq \lambda_n$.
If $\Lambda =\langle \X, \LL \rangle=\lambda_1X_1+\lambda_2X_2+\cdots +\lambda_n X_n$, then independence ensures that
$M_{\Lambda}(t) = \prod_{i=1}^n M(\lambda_i t)$. Thus, 
 $\cnorm{A}_{\X,d}^d=\E [\Lambda^d]/d! = [t^d] M_{\Lambda}(t)$. \qed

%%%%%%%%%%%%%%%%%%%%%%%%%%%%%%%%%%%%%%
\subsection{Proof of Theorem \ref{Theorem:Main}.c}
Maintain the same notation as in the proof of (b).  However, we only assume existence of the first $d$ moments of the $X_i$. In this case,  $M_{\Lambda}(t)$ is a formal series with $\kappa_1,\kappa_2,\ldots,\kappa_d$ determined and the remaining
cumulants treated as formal variables.  Then
\begin{align*}
M_{\Lambda}(t) 
&= \prod_{i=1}^n M(\lambda_i t) 
\overset{\eqref{eq:CumulantGen}}{=} \exp\bigg( \sum_{i=1}^n K(\lambda_i t) \bigg)\\
&\overset{\eqref{eq:CumulantGen}}{=} \exp\bigg( \sum_{j=1}^{\infty} \kappa_j (\lambda_1^j +\lambda_2^j+\cdots +\lambda_n^j)\frac{t^j}{j!}\bigg) \\
&=\exp \bigg( \sum_{j=1}^{\infty} \kappa_j \tr (A^j)\frac{t^j}{j!}\bigg) \\
&\overset{\eqref{eq:ExpoBell}}{=}\sum_{\ell=0}^{\infty} B_{\ell}(\kappa_1\tr A, \kappa_2\tr A^2, \ldots, \kappa_{\ell}\tr A^{\ell})\frac{t^{\ell}}{\ell !}. 
\end{align*}
Expanding the right side of \eqref{eq:ExpoBell} yields
\begin{equation}\label{eq:Bxy}
B_{\ell}(x_1,x_2, \ldots, x_{\ell})=\ell !\sum_{\substack{j_1,j_2, \ldots, j_{\ell}\geq 0\\ j_1+2j_2+\cdots +\ell j_{\ell}=\ell  }}\prod_{r=1}^{\ell} \frac{x_r^{j_r}}{(r!)^{j_r} j_r!}=\ell !\sum_{\vec{\pi}\vdash \ell}\frac{x_{\vec{\pi}}}{y_{\vec{\pi}}},
\end{equation} 
in which $x_{\vec{\pi}}=x_{i_1}x_{i_2}\cdots x_{i_j}$ for a each partition $\vec{\pi}=(i_1, i_2, \ldots, i_j)$ of $\ell$. 
Substitute $x_i= \kappa_i \tr (A^i)$ above and obtain
\begin{equation*}
d!\cnorm{A}_{\X,d}^d=[t^d] M_{\Lambda}(t) = B_{d}(\kappa_1\tr A, \kappa_2\tr A^2, \ldots, \kappa_d\tr A^d). 
\end{equation*}
Finally, \eqref{eq:Bxy} and the above ensure that
\begin{equation*}
\cnorm{A}_{\X,d}^d = \sum_{\vec{\pi}\vdash d}\frac{\kappa_{\vec{\pi}}p_{\vec{\pi}}}{y_{\vec{\pi}}}
\quad \text{for $A \in \h_n$}.\qedhere
\end{equation*}

%%%%%%%%%%%%%%%%%%%%%%%%%%%%%%%%%%%%%%

\subsection{Proof of Theorem \ref{Theorem:Main}.d}
Recall that a convex function $f:\R^n\to \R$ is  Schur convex if and only if it is symmetric \cite[p.~258]{Roberts}. Suppose that $\vec{x},\vec{y}\in \R^n$. Let $\X=(X_1, X_2, \ldots, X_n)$ be a random vector, in which $X_1, X_2, \ldots, X_n \in L^d(\Omega, \mathcal{F}, \P)$ are identically distributed. Define random variables $X=\langle \X, \vec{x}\rangle$ and  $Y=\langle \X, \vec{y}\rangle$. 

Define $\NN:\R^n\to \R_{\geq 0}$ by $\NN(\vec{x})=\Big(\frac{\E  |\langle \X, \vec{x}\rangle|^d}{\Gamma(d+1)}\Big)^{1/d}$. This function satisfies 
\begin{equation*}
\NN(\vec{x}+\vec{y})=\bigg(\frac{\E  |\langle \X, \vec{x}+\vec{y}\rangle|^d}{\Gamma(d+1)}\bigg)^{1/d}=\bigg(\frac{\E  |X+Y|^d}{\Gamma(d+1)}\bigg)^{1/d}\leq \NN(\vec{x}) +\NN(\vec{y})
\end{equation*} 
as seen in the proof of Theorem \ref{Theorem:Main}.a. Homogeneity implies that $\NN$ is convex on $\R^n$. Finally,  $\NN$ is symmetric since the random variables $X_1, X_2, \ldots, X_n$ are identically distributed. It follows that  $\NN$ is Schur convex. Therefore, 
$\vec{\lambda}(A)\mapsto \NN(\lambda_1, \lambda_2, \ldots, \lambda_n)=\cnorm{A}_{X,d}$
is Schur convex.

\begin{remark}\label{Remark:SchurNot}
Note that independence is not required in the previous argument.
\end{remark}

%%%%%%%%%%%%%%%%%%%%%%%%%%%%%%%%%%%%%%
\subsection{Proof of Theorem \ref{Theorem:Main}.e}\label{Subsection:ProofE}
The initial details of the proof parallel those of \cite[Thm.~3]{Aguilar}.
Let $\V$ be a $\C$-vector space with conjugate-linear involution $*$
and suppose that the real-linear subspace $\V_\R=\{v\in \V:v=v^*\}$ of $*$-fixed points has the norm $\| \cdot\|$.
Then $e^{it}v+e^{-it}v^* \in \V_{\R}$ for each $v \in \V$ and $t \in \R$, and
the path $t\mapsto \|e^{it}v+e^{-it}v^*\|$ is continuous for each $v\in \V$.
The following is \cite[Prop.~15]{Aguilar}.

\begin{lemma}\label{Lemma:GenNorm}
For even $d \geq 2$, the following 
is a norm on $\V$ that extends $\|\cdot\|$:
\begin{equation}\label{eq:ExtendRealNorm}
\NN_d(v)= \bigg( \frac{1}{2\pi \binom{d}{d/2}}\int_0^{2\pi}\|e^{it}v+e^{-it}v^*\|^d\,dt \bigg)^{1/d}.
\end{equation}
\end{lemma}

Let $\mx$ be the free monoid generated by $x$ and $x^*$. 
Let $|w|$ denote the length of a word $w\in\mx$ 
and let $|w|_x$ count the occurrences of $x$ in $w$.  
For $Z\in \M_n$, let $w(Z)\in \M_n$ be the natural evaluation of $w$ at $Z$.
For example, if $w = xx^*x^2$, then $|w| = 4$, $|w|_x = 3$, and $w(Z) = Z Z^* Z^2$.
The next lemma is \cite[Lem.~16]{Aguilar}.

\begin{lemma}\label{Lemma:Expr}
Let $d\geq 2$ be even and let $\vec{\pi}=(\pi_1,\pi_2,\ldots,\pi_r)$ be a partition of $d$. 
If $Z\in \M_n$, then
\begin{equation}\label{eq:int2poly}
\begin{split}
&\frac{1}{2\pi} \int_0^{2\pi}\tr(e^{it} Z+e^{-it} Z^*)^{\pi_1}\cdots \tr(e^{it}Z+e^{-it}Z^*)^{\pi_r}\,dt \\
&\qquad\qquad = \sum_{\substack{
		w_1,\ldots,w_r \in \mx \colon \\
		|w_j|=\pi_j\ \forall j \\
		|w_1\cdots w_r|_x = \frac{d}{2}
}} \tr w_1(Z)\cdots\tr w_r(Z).
\end{split}
\end{equation}
\end{lemma}

Given a partition $\vec{\pi}=(\pi_1,\pi_2,\ldots,\pi_r)$ of $d$ and $Z\in \M_n$ let 
\begin{equation}\label{eq:TDef}
\TT_{\vec{\pi}}(Z)=\frac{1}{\binom{d}{d/2}}\sum_{\substack{
	w_1,\ldots,w_r \in \mx \colon \\
	|w_j|=\pi_j\ \forall j \\
	|w_1\cdots w_r|_x = \frac{d}{2}
}} \tr w_1(Z)\cdots\tr w_r(Z),
\end{equation}
that is, $\TT_{\vec{\pi}}(Z)$ is $1/{d\choose d/2}$ times the sum over the $\binom{d}{d/2}$ 
possible locations to place $d/2$ adjoints ${}^*$ among the $d$ copies of $Z$ in
\begin{equation*}
(\tr \underbrace{ZZ\cdots Z}_{\pi_1})
(\tr \underbrace{ZZ\cdots Z}_{\pi_2})
\cdots
(\tr \underbrace{ZZ\cdots Z}_{\pi_r}).
\end{equation*}

Consider the conjugate transpose $*$ on $\V=\M_n$.
The corresponding real subspace of $*$-fixed points is $\V_{\R} = \h_n$. 
Apply Proposition \ref{Lemma:GenNorm} to the norm $\cnorm{\cdot}_d$ on $\h_n$ and 
obtain the extension $\NN_d(\cdot)$ to $\M_n$ defined by \eqref{eq:ExtendRealNorm}.

If $Z \in \M_n$ and $\NN_d(A) = \|A\|_d$ is the norm for $A \in \h_n$,
then Proposition \ref{Lemma:GenNorm} ensures that the following is a norm on $\M_n$:
\begin{align*}
\NN_d(Z)
&\overset{\eqref{eq:ExtendRealNorm}}{=} \bigg( \frac{1}{2\pi \binom{d}{d/2}} \int_0^{2 \pi} \cnorm{ e^{it} Z + e^{-it}Z}_{\vec{X},d}^d \,dt\bigg)^{1/d} \\
&\overset{\eqref{eq:RealPermForm}}{=} \bigg( \frac{1}{2\pi \binom{d}{d/2}} \int_0^{2 \pi} \sum_{\vec{\pi} \,\vdash\, d} \frac{\kappa_{\vec{\pi}}p_{\vec{\pi}}( \vec{\lambda}(e^{it} Z  + e^{-it}Z^*))}{y_{\pi}}\,dt \bigg)^{1/d} \\
&\overset{\eqref{eq:pTrace}}{=} \bigg( \frac{1}{\binom{d}{d/2}} \sum_{\vec{\pi} \,\vdash\, d} \frac{\kappa_{\vec{\pi}}}{y_{\vec{\pi}}} \cdot\frac{1}{2\pi} \int_0^{2\pi} \tr(e^{it} Z + e^{-it}Z^*)^{\pi_1} \cdots
\tr (e^{it} Z + e^{-it} Z^*)^{\pi_r} \,dt \bigg)^{1/d} \\
&\overset{\eqref{eq:int2poly}}{=} \bigg( \frac{1}{\binom{d}{d/2}} \sum_{\vec{\pi} \,\vdash\, d} \frac{ \kappa_{\vec{\pi}} }{y_{\vec{\pi}}} \sum_{\substack{
		w_1,\ldots,w_r \in \mx \colon \\
		|w_j|=\pi_j\ \forall j \\
		|w_1\cdots w_r|_x = \frac{d}{2}
}} \tr w_1(Z)\cdots\tr w_r(Z)\bigg)^{1/d} \\
%&\overset{???}{=} \left( \frac{1}{\binom{d}{d/2}} \sum_{\vec{\pi} \,\vdash\, d} \frac{ \kappa_{\vec{\pi}}\TT_{\vec{\pi}}(Z) \binom{d}{d/2} }{y_{\vec{\pi}}} \right)^{1/d} \\
&\overset{\eqref{eq:TDef}}{=} \bigg(  \sum_{\vec{\pi} \,\vdash\, d} \frac{ \kappa_{\vec{\pi}}\TT_{\vec{\pi}}(Z)  }{y_{\vec{\pi}}} \bigg)^{1/d}.
\qed
\end{align*}

%%%%%%%%%%%%%%%%%%%%%%%%%%%%%%%%%%%%%%%%%%%%
\section{Open Questions}\label{Section:Questions}

If $\norm{\cdot}$ is a norm on $\M_n$, then there is a scalar multiple of it (which may depend upon $n$)
that is submultiplicative.  One wonders which of the norms $\cnorm{\cdot}_{\X,d}$ are submultiplicative, 
or perhaps are when multiplied by a constant independent of $n$.  For example, \eqref{eq:cn2}
ensures that for $d=2$, a mean-zero distribution leads to a multiple of the Frobenius norm.  If $\mu_2 = 2$,
then the norm is submultiplicative.

\begin{problem}
Characterize those $\X$ that give rise to submultiplicative norms.
\end{problem}

For the standard exponential distribution, \cite[Thm~31]{Aguilar} provides an answer to the next question.
An answer to the question in the general setting eludes us.

\begin{problem}
Characterize the norms $\cnorm{\cdot}_{\X,d}$ that arise from an inner product.
\end{problem}

Several other unsolved questions come to mind.

\begin{problem}
Identify the extreme points with respect to random vector norms.
\end{problem}

\begin{problem}
Characterize norms on $\M_n$ or $\h_n$ that arise from random vectors.
\end{problem}

%%%%%%%%%%%%%%%%%%%%%%%%%%%%%%%%%%%%%%
%%%%%%%%%%%%%%%%%%%%%%%%%%%%%%%%%%%%%%
\bibliography{RandomNorms}

\providecommand{\bysame}{\leavevmode\hbox to3em{\hrulefill}\thinspace}
\providecommand{\MR}{\relax\ifhmode\unskip\space\fi MR }
% \MRhref is called by the amsart/book/proc definition of \MR.
\providecommand{\MRhref}[2]{%
  \href{http://www.ams.org/mathscinet-getitem?mr=#1}{#2}
}
\providecommand{\href}[2]{#2}
\begin{thebibliography}{10}

\bibitem{Aguilar}
Konrad Aguilar, \'Angel Ch\'avez, Stephan~Ramon Garcia, and Jurij
  Vol\v{c}i\v{c}, \emph{Norms on complex matrices induced by complete
  homogeneous symmetric polynomials}, Bulletin of the London Mathematical
  Society, in press. \url{https://doi.org/10.1112/blms.12679}.

\bibitem{Barvinok}
A.~I. Barvinok, \emph{Low rank approximations of symmetric polynomials and
  asymptotic counting of contingency tables},
  \url{https://arxiv.org/abs/math/0503170}.

\bibitem{Baston}
V.~J. Baston, \emph{Two inequalities for the complete symmetric functions},
  Math. Proc. Cambridge Philos. Soc. \textbf{84} (1978), no.~1, 1--3.
  \MR{485422}

\bibitem{Bell}
E.~T. Bell, \emph{Exponential polynomials}, Ann. of Math. (2) \textbf{35}
  (1934), no.~2, 258--277. \MR{1503161}

\bibitem{Billingsley}
Patrick Billingsley, \emph{Probability and measure}, Wiley Series in
  Probability and Statistics, John Wiley \& Sons, Inc., Hoboken, NJ, 2012,
  Anniversary edition [of MR1324786], With a foreword by Steve Lalley and a
  brief biography of Billingsley by Steve Koppes. \MR{2893652}

\bibitem{BGON}
Albrecht B\"{o}ttcher, Stephan~Ramon Garcia, Mohamed Omar, and Christopher
  O'Neill, \emph{Weighted means of {B}-splines, positivity of divided
  differences, and complete homogeneous symmetric polynomials}, Linear Algebra
  Appl. \textbf{608} (2021), 68--83. \MR{4140644}

\bibitem{ENT1}
Alexandros Eskenazis, Piotr Nayar, and Tomasz Tkocz, \emph{Gaussian mixtures:
  entropy and geometric inequalities}, Ann. Probab. \textbf{46} (2018), no.~5,
  2908--2945. \MR{3846841}

\bibitem{ENT2}
\bysame, \emph{Sharp comparison of moments and the log-concave moment problem},
  Adv. Math. \textbf{334} (2018), 389--416. \MR{3828740}

\bibitem{GOOY}
Stephan~Ramon Garcia, Mohamed Omar, Christopher O'Neill, and Samuel Yih,
  \emph{Factorization length distribution for affine semigroups {II}:
  asymptotic behavior for numerical semigroups with arbitrarily many
  generators}, J. Combin. Theory Ser. A \textbf{178} (2021), 105358, 34.
  \MR{4175889}

\bibitem{Gould}
H.~W. Gould, \emph{Explicit formulas for {B}ernoulli numbers}, Amer. Math.
  Monthly \textbf{79} (1972), 44--51. \MR{306102}

\bibitem{Haagerup}
Uffe Haagerup, \emph{The best constants in the {K}hintchine inequality}, Studia
  Math. \textbf{70} (1981), no.~3, 231--283 (1982). \MR{654838}

\bibitem{Havrilla}
Alex Havrilla and Tomasz Tkocz, \emph{Sharp {K}hinchin-type inequalities for
  symmetric discrete uniform random variables}, Israel J. Math. \textbf{246}
  (2021), no.~1, 281--297. \MR{4358280}

\bibitem{HJ}
Roger~A. Horn and Charles~R. Johnson, \emph{Matrix analysis}, second ed.,
  Cambridge University Press, Cambridge, 2013. \MR{2978290}

\bibitem{Hunter}
D.~B. Hunter, \emph{The positive-definiteness of the complete symmetric
  functions of even order}, Math. Proc. Cambridge Philos. Soc. \textbf{82}
  (1977), no.~2, 255--258. \MR{450079}

\bibitem{LO}
Rafa\l Lata\l~a and Krzysztof Oleszkiewicz, \emph{A note on sums of independent
  uniformly distributed random variables}, Colloq. Math. \textbf{68} (1995),
  no.~2, 197--206. \MR{1321042}

\bibitem{LewisHermitian}
A.~S. Lewis, \emph{Convex analysis on the {H}ermitian matrices}, SIAM J. Optim.
  \textbf{6} (1996), no.~1, 164--177. \MR{1377729}

\bibitem{LewisGroup}
\bysame, \emph{Group invariance and convex matrix analysis}, SIAM J. Matrix
  Anal. Appl. \textbf{17} (1996), no.~4, 927--949. \MR{1410709}

\bibitem{Roberts}
A.~Wayne Roberts and Dale~E. Varberg, \emph{Convex functions}, Pure and Applied
  Mathematics, Vol. 57, Academic Press [Harcourt Brace Jovanovich, Publishers],
  New York-London, 1973. \MR{0442824}

\bibitem{Roventa}
Ionel Roven\c{t}a and Lauren\c{t}iu~Emanuel Temereanc\u{a}, \emph{A note on the
  positivity of the even degree complete homogeneous symmetric polynomials},
  Mediterr. J. Math. \textbf{16} (2019), no.~1, Paper No. 1, 16. \MR{3887204}

\bibitem{StanleyBook1}
Richard~P. Stanley, \emph{Enumerative combinatorics. {V}ol. 1}, Cambridge
  Studies in Advanced Mathematics, vol.~49, Cambridge University Press,
  Cambridge, 1997, With a foreword by Gian-Carlo Rota, Corrected reprint of the
  1986 original. \MR{1442260}

\bibitem{StanleyBook2}
\bysame, \emph{Enumerative combinatorics. {V}ol. 2}, Cambridge Studies in
  Advanced Mathematics, vol.~62, Cambridge University Press, Cambridge, 1999,
  With a foreword by Gian-Carlo Rota and appendix 1 by Sergey Fomin.
  \MR{1676282}

\bibitem{Tao}
Terence Tao, \emph{Schur convexity and positive definiteness of the even degree
  complete homogeneous symmetric polynomials},
  \url{https://terrytao.wordpress.com/2017/08/06/}.

\end{thebibliography}
\bibliographystyle{amsplain}

\end{document}